\documentclass[11pt]{article}
\usepackage{geometry}                
\usepackage{graphicx}
\usepackage{amssymb, amsthm,amsmath}
\usepackage{epstopdf}
\usepackage[vlined, ruled, linesnumbered, norelsize]{algorithm2e}

\newtheorem{theorem}{Theorem}[section]
\newtheorem*{theorem*}{Theorem}
\newtheorem{lemma}[theorem]{Lemma}
\newtheorem{proposition}[theorem]{Proposition}
\newtheorem{corollary}[theorem]{Corollary}
\newtheorem{conjecture}[theorem]{Conjecture}

\theoremstyle{definition}
\newtheorem{definition}[theorem]{Definition}

\renewcommand{\P}{ \mathcal{P}}
\newcommand{\N}{ \mathbb{N}}

\newcommand{\ord}{ {\rm ord}}

\title{Counting Composites with  Two Strong Liars}
\author{Eric Bach \thanks{Research supported by NSF: CCF-0635355 and 
ARO: W911NF9010439} \\ 
University of Wisconsin--Madison \\
1210 W. Dayton St. \\ 
Madison, WI 53706 \\
bach@cs.wisc.edu
 \and Andrew Shallue  \\
 Illinois Wesleyan University \\
 1312 Park St. \\
 Bloomington, IL  61701 \\
 ashallue@iwu.edu
 }
\date{}                                           

\begin{document}
\maketitle

\begin{abstract}
The strong probable primality test is an important practical tool for discovering 
prime numbers.  Its effectiveness derives from the following fact:  for any odd
composite number $n$, if a base $a$ is chosen at random, the algorithm
is unlikely to claim that $n$ is prime.  If this does happen we call $a$ a liar.
In 1986, Erd\H{o}s and Pomerance computed the normal and average number of liars, 
over all $n \leq x$.  We continue this theme and use a variety of techniques to count 
$n \leq x$ with exactly two strong liars, those being the $n$ for which the strong test 
is maximally effective.  We evaluate this count asymptotically and give an improved 
algorithm to determine it exactly.  We also provide asymptotic counts for the restricted 
case in which $n$ has two prime factors, and for the $n$ with exactly two Euler liars.
\end{abstract}

\section{Introduction}

The strong probable primality test (studied by Selfridge, Miller, Rabin, and others)
is an important tool for discovering prime numbers 
in practice.  Its success relies on the scarcity of strong liars.

\begin{definition}\label{def:strongliar}
Let $n$ be an odd composite integer.  Write $n-1$ as $2^k \cdot n'$ where $n'$ 
is the odd part of $n-1$ and $k = {\rm ord}_2(n)$.  
Then $a$ is a strong liar with respect to $n$
if either 
\begin{enumerate} \item $a^{2^i n'} \equiv -1 \bmod{n}$ for 
some $0 \leq i < k$ or
\item $a^{n'} \equiv 1 \bmod{n}$.
\end{enumerate}
\end{definition}

Throughout we will use this convention of writing $n-1$ as $2^k \cdot n'$ where $n'$ is odd.
We use $\log$ for the natural logarithm, and $\varphi(n)$ for the count of 
$1 \leq a \leq n$ with $\gcd(a,n)=1$.  When using asymptotic notation, implied constants
with subscripts depend on that variable.  We use the Euler constant $\gamma$
defined by
$$
\gamma = \int_1^{\infty}  \frac{1}{\lfloor x \rfloor} - \frac{1}{x} \ {\rm d}x 
\approx 0.5772 \enspace .
$$

If $n$ is an odd prime then the condition in Definition \ref{def:strongliar} holds 
for all $a$ not divisible by $n$.  If $n$ is even then $n' = n-1$ is odd and the only  
strong liars are $a$ such that  $a^{n-1} \equiv \pm 1 \mod{n}$.  
While it is possible to have strong liars 
in this case, we restrict to $n$ odd since it is more interesting for primality testing.
Nevertheless, for convenience we define the set 
$$
S(n) = \{ a \mod{n} \ : \ a^{n'} \equiv 1 \bmod{n} \mbox{ or } a^{2^i n'} \equiv -1 \bmod{n}
\mbox{ for some $0 \leq i < k$} \} 
$$
for general $n$.
As a shorthand we will refer to elements of $S(n)$ as strong liars, even though
if $n$ prime then $a \in S(n)$ is truthfully giving us evidence that $n$ is prime.
Though they are not our main focus, it is useful to define two other types of liars.

\begin{definition}
Let $n$ be an odd composite integer and $(a \mid n)$ be the Jacobi symbol.
Then $a$ is a Fermat liar with respect to $n$ if $a^{n-1} \equiv 1 \mod{n}$
and $a$ is an Euler liar with respect to $n$ if 
$\gcd(a,n)=1$ and $a^{(n-1)/2} \equiv (a \mid n) \mod{n}$.
\end{definition}

We similarly define 
\begin{align*} 
F(n) &= \{ a \mod{n} \ : \ a^{n-1} \equiv 1 \mod{n} \} \\
E(n) &= \{ a \mod{n} \ : \  \gcd(a,n)=1 \mbox{ and } a^{(n-1)/2} \equiv ( a \mid n) \mod{n} \} \enspace ,
\end{align*}
and an important fact is that $S(n) \subseteq E(n) \subseteq F(n)$.  While $E(n)$ and $F(n)$ 
are always subgroups of the group of units modulo $n$, $S(n)$ may not be.

Our primary interest will be in counting $n$ where $|S(n)|$ is an extremal value.  On the practical 
side, it is useful to know how often we might expect the strong primality test to be as 
effective as possible or as ineffective as possible.  
In Section \ref{sec:facts} we discuss what is known about the worst case, but our new 
contribution involves counting best case composites.  This occurs when $n > 3$ has two strong liars, 
and the happy consequence is that  one trial of the strong primality test
 is sufficient to prove compositeness.

For theoretical motivation, we will see that 
a key quantity is 
$$
\prod_{p \mid n} \gcd(p',n')
$$
which is interesting in its own right.  Finally, our work is complimentary to that of 
Erd\H{o}s and Pomerance in \cite{ErdPom86}, who provide upper and lower bounds on 
the arithmetic and geometric mean of all three sets $S(n)$, $E(n)$ and $F(n)$.
They also provide some discussion of counts of $n$ with extremal values of $F(n)$, 
one result of which we extend to $S(n)$.

In addressing these questions, we prove results using both analytic and algorithmic 
techniques.  Our main result is the following.

\begin{theorem*} \label{thm:minliars}
The number of odd $n \leq x$ with exactly two strong liars is given by 
$$
(1 + o(1)) \frac{x e^{-\gamma}}{\log\log\log{x}}
$$
where $\gamma$ is Euler's constant.
\end{theorem*}

We also prove that the number of odd $n \leq x$ with exactly two Euler liars 
is half that amount.
In \cite[Section 6]{ErdPom86} it is noted that the number of $n \leq x$ with 
$F(n) = 1$ follows the same asymptotic formula as the count of odd $n \leq x$
with two strong liars.  All of these results utilize an argument 
from \cite{Erd48}, where Erd\H{o}s proves
the number of $n \leq x$ with $\gcd(n, \varphi(n)) = 1$ is 
also $(1+o(1))xe^{-\gamma}/\log\log\log{x}$.

It would be interesting to know 
how many $n$ with two strong liars have $r$ prime factors.  A start on that project 
is the following theorem.

\begin{theorem*} \label{thm:twoprimes}
The number of odd $n \leq x$ with $n = pq$, $p,q$ both prime and  $\gcd(p',q')=1$ is
\begin{align*}
& = (1 + o(1)) \frac{C x \log\log{x}}{\log{x}} \\
 \mbox{where   } C &:= \prod_{p > 2} \left( 1 - \frac{1}{(p-1)^2} \right) = 0.66016 \dots
\end{align*}
is the Hardy-Littlewood twin prime constant.  If $C$ is replaced by $3C/4$, we get the count 
of odd $n \leq x$ with two prime factors and $|S(n)| = 2$.
\end{theorem*}

We have also proven an asymptotic formula for the number of $n = pq$ with $p,q \leq x$
and $\gcd(p',q')=1$, but will not address that result here.

For intuition on these two theorems, note that Mertens' theorem \cite[Theorem 429]{HarWri79}
gives us 
$$
\prod_{p \leq \log\log{x}} \left(1 - \frac{1}{p} \right)
= (1 + o(1)) \frac{e^{-\gamma}}{\log\log\log{x}}
$$
and that the number of positive integers $n \leq x$ that are the product of two primes 
is asymptotic to $x (\log\log{x})/\log{x}$ \cite{Lan00}.  So the count in the first theorem
 is driven by sieving by primes less than $\log\log{x}$, while  the number of $n \leq x$
 with two prime factors drives the second theorem.

Finally, we have designed a new algorithm that exactly counts the number of odd $n \leq x$ 
with two prime factors, and does so more quickly than simply applying a known formula 
to each $n$.

\begin{theorem*} \label{thm:alg}
There is an algorithm that, given $x$, computes the number of positive integers $n \leq x$
with  two strong liars.  This algorithm requires $O(x (\log{x})(\log\log{x}))$ bit operations
and $O(x \log{x})$ space.
\end{theorem*}

\section{Facts about strong liars} \label{sec:facts}

We collect in this section a number of facts related to strong liars, some of which
will be useful for results in later sections.  Recall that $n'$ is the odd part of $n-1$, 
so that for example $10' = 9$ and $9' = 1$.
We use $k$ to denote $\ord_2(n-1)$, so that $n-1 = 2^k \cdot n'$.

First, note that if $n$ is odd then $\pm 1$ are always strong liars, since
$1^{n'} \equiv 1 \mod{n}$ for all $n$ and $(-1)^{n'} \equiv -1 \mod{n}$ for odd $n$.
In fact, if $n$ is odd then $a \in S(n)$ implies $-a \in S(n)$.  
For if $a^{n'} \equiv \pm 1 \mod{n}$ then $(-a)^{n'} \equiv \mp 1 \mod{n}$ which makes $-a$ 
a strong liar.  And if $a^{2^i n'} \equiv -1 \mod{n}$ for some $1 \leq i < k$, then 
$(-a)^{2^i n'} = (-1)^{2^i n'} \cdot a^{2^i n'} \equiv -1 \mod{n}$, which again makes $-a$
a strong liar.  Altogether, we see that if $n$ is odd, then $|S(n)|$ is even and at least $2$.
This means we can restrict the strong test to choosing $1 < a < (n-1)/2$ with no loss.

More generally, we would like an explicit formula for the size of $S(n)$.  This 
was accomplished by Monier.

\begin{proposition}[\cite{Mon80}] \label{prop:monier}
For $n$ any positive integer, let $n'$ be the odd part of $n-1$ and let $r$
be the number of distinct prime divisors of $n$.  Let $v = {\rm min}_{p \mid n} {\rm ord}_2(p - 1)$.
Then
$$
|S(n)| = \left(1 + \frac{2^{rv}-1}{2^r-1}\right)
\prod_{p \mid n} \gcd(n',p') \enspace .
$$
\end{proposition}

It quickly follows that $|S(n)| \leq \varphi(n)/4$ when $n > 9$ is odd and composite.
  Thus by performing $\log_2{(\frac{1}{\sqrt \epsilon})}$ 
independent trials we can lower the probability that $n$ is a composite falsely 
reported as prime to below $\epsilon$.
Note that if $n$ is prime, Proposition \ref{prop:monier} correctly gives
$|S(n)| = (1 + 2^k - 1) \gcd(n', n') = n-1$.

We briefly address the worst case, i.e. composite $n > 9$ for which 
$|S(n)|$ reaches the maximum of $\varphi(n)/4$.  Such $n$ are fairly easy 
to characterize, if not quite so easy to count.  Consider the following theorem.

\begin{theorem}[\cite{DamLadPom93}] \label{thm:worst}
Let $C_3$ be the set of odd, composite integers $n$ with $|S(n)| > \varphi(n)/8$.
Then $C_3$ is composed of the following:
\begin{enumerate}
\item[(1)] $(m+1)(2m+1)$, where $m+1$, $2m+1$ are odd primes, 
\item[(2)] $(m+1)(3m+1)$, where $m+1$, $3m+1$ are primes congruent to $3 \mod{4}$,
\item[(3)] Carmichael numbers $n$ with three prime factors where there exists integer $s$
with $2^s$ exactly dividing $p-1$ for all $p \mid n$,
\item[(4)] $9, 25, 49$.
\end{enumerate}
\end{theorem}

By Proposition \ref{prop:monier}, it follows that the $n$ with $|S(n)| = \varphi(n)/4$ are 
exactly $n$ in case (1) with $2 \| m$ and Carmichael numbers in case (3) whose
three prime factors are all congruent 
to $3$ modulo $4$ (this also appears in the proof to Theorem \ref{thm:worst}).

Unfortunately, an asymptotic formula for either case remains elusive.  Nor has it been proven
that there are infinitely many integers in either case.  On the other hand, 
infinitely many $n$ of the form 
$(m+1)(2m+1)$ would follow from the strong prime tuples
conjecture \cite{BateHorn62}, and there is a precise conjecture on the number 
of Carmichael numbers with three prime factors.  

\begin{conjecture}[\cite{GranPom02}]
The number of Carmichael numbers with three prime factors is asymptotic to 
$$
C \frac{x^{1/3}}{\log^3{x}} \enspace ,
$$
where $C$ is an absolute constant that can be given precisely.
\end{conjecture}

With the help of Theorem \ref{thm:twostrongliars},
it can be shown that $n$ with $|S(n)| = 2$ are much more common 
than $n$ with $|S(n)| = \varphi(n)/4$.

We now shift to counting odd $n$ with exactly two strong liars. 
The following characterization will be useful.

\begin{proposition}\label{prop:n/p}
Suppose $n$ is odd and composite.  Then $|S(n)| = 2$ 
 if and only if 
1) $n$ is  divisible by $p \equiv 3 \bmod{4}$ and 
2) $\gcd(p', (n/p)') = 1$ for all primes $p$ dividing $n$.
\end{proposition}
\begin{proof}
 First notice that 
$$
(p-1)(n/p - 1) = n-1 - (n/p - 1) - (p-1) \enspace ,
$$
so that $\gcd(p', n') = \gcd(p', (n/p)')$.

Thus if $\gcd(p',(n/p)')=1$ for all $p \mid n$ the product term in Monier's formula is $1$.
If $n$ is odd and divisible by $p \equiv 3 \bmod{4}$ then $v=1$ and we conclude 
that $L(n)=2$.  

If instead we assume $L(n)=2$ then $1 + \frac{2^{rv}-1}{2^r-1} \leq 2$.
  If $n$ is odd then 
$1 + \frac{2^{rv}-1}{2^r-1} \geq 2$, with equality only if $v=1$.  Thus 
$n$ is divisible by a prime congruent to $3$ modulo $4$ and 
$\prod_{1\leq i \leq r} \gcd(n',p_i') = \prod_{1\leq i \leq r} \gcd(n',(n/p_i)') = 1$.
\end{proof}

Monier  also proved a formula for Euler liars.

\begin{proposition}[\cite{Mon80}]\label{prop:eulerliars}
Let $n$ be odd.  Define $e(n) = \prod_{p \mid n} \gcd( \frac{n-1}{2}, p-1)$
and 
$$
\delta(n) = \left\{ \begin{array}{ll} 2 & \mbox{ if $v = \ord_2(n-1)$  } \\
 1/2 & \mbox{ if there is $p \mid n$ with $\ord_2(p-1) < \ord_2(n-1)$ and $\ord_p(n)$ odd} \\
 1 & \mbox{ otherwise, i.e. $\ord_p(n)$ even for all $p \mid n$ with $\ord_2(p-1) < \ord_2(n-1)$} \end{array} \right .
 $$
 Then $|E(n)| = \delta(n) \cdot e(n)$.
\end{proposition}

If $n$ is odd then the minimum number of Euler liars is $2$ since $\pm 1$ 
are always Euler liars.  

\begin{proposition} \label{prop:eulercases}
Suppose $n$ is odd and composite.  Then $|E(n)| = 2$ if and only if
\begin{enumerate}
\item $n \equiv 3 \mod{4}$ and $\prod_{p \mid n} \gcd(p', n') = 1$, or
\item $n \equiv 1 \mod{4}$ with $n = pq$, $p, q \equiv 3 \mod{4}$, 
and $\gcd(p',n') \cdot \gcd(q',n') = 1$.
\end{enumerate}

\end{proposition}
\begin{proof}
First suppose that $\prod_{p \mid n} \gcd(p', n') = 1$.  If $n \equiv 3 \mod{4}$
then $\gcd(\frac{n-1}{2}, p-1) = 1$ for all $p \mid n$.  Additionally, $\delta(n) = 2$
since $\ord_2(p-1)$ cannot be any smaller.  If instead $n=pq$ with $p, q \equiv 3 \mod{4}$, 
then $\ord_2(n-1) = 2$ and so $\delta(n) = 1/2$, while 
$\prod_{p \mid n} \gcd(\frac{n-1}{2}, p-1) = 4$.

Now suppose that $|E(n)| = 2$.  It is impossible to have 
$e(n) = 2$.  For if $n \equiv 3 \mod{4}$ then the product 
will be odd, while if $n \equiv 1 \mod{4}$, $n$ odd means the product will be divisible by 
at least one factor of $2$ for each prime factor of $n$.  Thus the only two possibilities 
are 1) $\delta(n) = 2$ and $e(n) = 1$ and 2) $\delta(n) = 1/2$ and $e(n) = 4$.

In case 1), $e(n)=1$ implies $n \equiv 3 \mod{4}$ and $\prod_{p \mid n} \gcd(n',p')=1$, 
since otherwise $e(n)$ would be larger.  With $n \equiv 3 \mod{4}$, it must be divisible 
by a prime congruent to $3$ modulo $4$, and so it follows that $\delta(n) = 2$.

In case 2), $\delta(n) = 1/2$ implies $n \equiv 1 \mod{4}$ and divisible by a prime 
congruent to $3$ modulo $4$.  Then
$e(n)=4$ implies $\prod_{p \mid n} \gcd(n',p')=1$ and $n$ is the product 
of two distinct prime factors, for otherwise the power of $2$ dividing  $e(n)$ would 
be greater.
\end{proof}

\section{Preliminaries}

The proofs of our asymptotic formulas
will utilize a number of results from analytic number theory.
Our goal is to craft an account that is readable and self-contained, and 
hence will not necessarily include best-possible results.

One tool will be counts of primes in arithmetic progressions.
The classic result is the prime number theorem for arithmetic progressions.

\begin{lemma} \label{PNTprogs}
If $\gcd(d,a)=1$, let $\pi(x,d,a)$ denote the number of primes
$\le x$ that are congruent to $a$ mod $d$.  Then
$$
\pi(x,d,a) = (1 + o_d(1)) \frac {x} {\varphi(d)\log x}.
$$
\end{lemma}

Next we have a version of the Brun-Titchmarsh inequality
from Montgomery and Vaughan \cite{MonVau73}.
Note that the constant is absolute for arbitrary $d$ smaller than $x$.

\begin{lemma} \label{bruntitch}
For $x > d \ge 1$, we have
$$
\pi(x,d,a) < \frac{2x} {\varphi(d) \log (x/d)} \enspace .
$$
\end{lemma}

The Siegel-Walfisz theorem \cite{Wal36}
gives an absolute lower bound, but the range of possible 
$d$ is much smaller.  

\begin{lemma} \label{siewal}
Assume $d \leq \log{x}$.  Then
$$
\pi(x,d,a) = (1 + o(1))\frac{x}{ \varphi(d) \log{x}} \enspace .
$$
\end{lemma}

We will rely on a number of prime reciprocal sums.
The most basic is a result of Landau in \cite[v. 1, p. 197]{Lan09}.

\begin{lemma} \label{landau}
We have
$$
\sum_{p \leq x}\frac{1}{p} = \log\log{x} + A + O((\log{x})^{-1})
$$
where $A$ is an absolute constant.
\end{lemma}

Bounds on $\pi(x,d,a)$ lead to asymptotic formulas for prime reciprocal 
sums over arithmetic progressions.  It is doubtful the following lemma 
is new, but a good reference is elusive.

\begin{lemma} \label{reciprocalsums}
Let $P(x,d)$ be the prime reciprocal sum over a particular arithmetic progression.  That is, 
$$
P(x,d) = \sum_{\substack{p \leq x \\ p \equiv 1 (d)}} \frac{1}{p}
$$
where the sum is over primes.  Then 
\begin{enumerate}
\item for $1 < d \leq x$ we have $P(x,d) = (1 + o_d(1)) (\log\log{x})/\varphi(d)$,
\item for $1 < d \leq \log{x}$ we have $P(x,d) = (1 + o(1)) (\log\log{x})/\varphi(d)$, 
\item for $1 < d \leq \sqrt{x}$ we have $P(x,d) < 2 \varphi(d)^{-1}(\log\log{x} + O(1))$.
\end{enumerate}
\end{lemma}
\begin{proof}
Replacing the sum by a Stieltjes integral and integrating
by parts, we get
\begin{equation}\label{intbyparts}
\sum_{\substack{p \le x \\ p \equiv 1 (d)}} \frac 1 p
=
\frac { \pi(t,d,1) } {t} \Big|_{3^-}^{x}
+
\int_{3}^{ x} \frac {\pi(t,d,1) dt} {t^2} \enspace .
\end{equation}
The first two cases are easier.  For general $d \leq x$ we
apply Lemma \ref{PNTprogs}  to get
$$
\frac{\pi(x,d,1)}{x} + \int_3^x (1 + o_d(1)) \frac{1}{\varphi(d)t \log{t}} \ {\rm d}t
= \frac{1+o_d(1)}{\varphi(d) \log{x}} + O(1) + \frac{1 + o_d(1)}{\varphi(d)} \log\log{x}
$$
which is equivalent to $(1 + o_d(1)) (\log\log{x})/\varphi(d)$.  If $d \leq \log{x}$
we apply Lemma \ref{siewal} to get the same result, except that the constant 
in the $o(1)$ does not depend on $d$.

For part 3), the first term of (\ref{intbyparts}) is 
$$
\frac { \pi(x,d,1) } {x}
\le
\frac {2  x} { x \varphi(d) \log(x/d)}
=
O \left( \frac 1 {\varphi(d)} \right),
$$
by Lemma \ref{bruntitch} and the assumption that $d \leq \sqrt{x}$.
We wish to push the lower bound of the integral to $2d$, which at worst costs
us one term of the sum, and only if $d+1$ is prime.
 Using Lemma \ref{bruntitch} again, the revised integral is bounded by
$$
\int_{2d}^{  xd} \frac {\pi(t,d,1) dt} {t^2}
\le
\int_{2d}^{xd} \frac {2 dt} {\varphi(d) t \log (t/d)}
=
2 \int_{2}^{ x} \frac {du} {\varphi(d) u \log u} \enspace .
$$
This is $2 \varphi(d)^{-1} (\log\log x + O(1))$.
The lost term of the sum makes no difference, since $1/p = 1/(d+1) \leq \varphi(d)^{-1}$.
\end{proof}

Next we give a brief introduction to sieve theory; interested readers are encouraged 
to peruse \cite{Greaves01} or \cite{Hoo76}.
Sieve theory is a collection of results for estimating the number of ``survivors" that remain 
after we start with an interval (or other large set) and remove elements that satisfy 
congruence conditions.  Typically, the exact formula for the number of survivors is of 
exponential complexity, and so one seeks approximations that are easier to 
evaluate but still reasonably accurate.

We use $S(x, \mathcal{P})$ to denote the count of integers up to $x$ 
coprime to the elements of $\mathcal{P}$, where $\mathcal{P}$ is a 
set of primes.  When $\mathcal{P}$ is the set of primes up to $z$ we instead 
use $S(x,z)$, and we replace $x$ with $\mathcal{X}$ when our base set is a subset 
of the integers up to $x$.  Our first sieve is
the Legendre sieve, an exercise in keeping track of the errors 
from the Sieve of Eratosthenes. 
Note that $2^{|\P|}$ is the error term.

\begin{theorem}[Legendre sieve] \label{legendre}
Let $\mathcal{P}$ be a set of primes.
Then
$$
S(x, \P) \leq x \prod_{p \in \P} \left( 1 - \frac{1}{p} \right) + 2^{|\P|} \enspace .
$$
\end{theorem}
\begin{proof}
Let $P$ be the product of all primes in $\P$, and let $d$ be an arbitrary divisor.
Using inclusion-exclusion we obtain
$$
S(x, \P) = \sum_{d \mid P} \mu(d) \left \lfloor \frac{x}{d} \right \rfloor
\leq x \sum_{d \mid P} \frac{\mu(d)}{d} + 1
= x \prod_{ p \in \P} \left( 1 - \frac{1}{p} \right) + 2^{|\P|} \enspace .
$$
\end{proof}

\begin{corollary} \label{legendre_eq}
Let $\P$ be the set of primes up to $z$, where $z \leq \log{x}$.  Then 
$$
S(x, \P) = S(x,z) = (1 + o(1)) \frac{e^{-\gamma} x}{\log{z}}
$$
where $\gamma \approx 0.5772$ is Euler's constant.
\end{corollary}
\begin{proof}
The product term is given by Mertens' theorem \cite[Theorem 429]{HarWri79}.
For the error term note that for $z$ large enough
$$
2^{\pi(z)} \leq 2^{\frac{2z}{\log{z}}}
\leq 2^{ \frac{2\log{x}}{\log{z}}} = x^{\frac{2}{ (\log_2{e})(\log{z})}}
= o\left( \frac{x}{\log{z}} \right) \enspace .
$$
\end{proof}

Despite the logarithmic bound on $z$, Corollary \ref{legendre_eq}
will be strong enough to give the main term in Theorem \ref{thm:twostrongliars}.
If we are willing to settle for an upper bound, we can generalize 
the set of sieving primes.

\begin{corollary} \label{legendre_upper}
Let $\P$ be an arbitrary set of primes smaller than $\log{x}$.  Then
$$
S(x,\P) \leq (1+o(1)) \cdot x \cdot {\rm exp}\left(- \sum_{p \in \P} \frac{1}{p} \right) \enspace .
$$
\end{corollary}
\begin{proof}
By convexity we have $\log(1 - 1/p) \leq -1/p$.  The error term is dealt with in similar 
fashion to Corollary \ref{legendre_eq}.
\end{proof}

The following application of the Legendre sieve will be used in the next section.

\begin{corollary}[\cite{Erd48}] \label{legendre_leastprime}
Let $p \to \infty$, $x \to \infty$ with $p + \log{p} \leq \log{x}$.
 Denote by $C_p(x)$ the number of integers $n \leq x$
for which the least prime factor of $n$ is $p$.  Then
$$
C_p(x) = (1 + o(1)) \frac{x e^{-\gamma}}{p \log{p}} \enspace .
$$
\end{corollary}
\begin{proof}
Note the least prime factor of $n$ is $p$ if and only if $n/p$ is coprime to 
all primes smaller than $p$.  Since $p + \log{p} \leq \log{x}$ implies $p \leq \log{(x/p)}$, 
we apply Corollary \ref{legendre_eq} to obtain
$$
C_p(x) = S(x/p, p) = (1 + o(1)) \frac{e^{-\gamma} x}{p \log{p}} \enspace .
$$
\end{proof}

The Legendre sieve can also be extended to other initial sets.  For example, 
Let $\mathcal{X}$ be the set of integers $n \leq x$ that are congruent to $3$ modulo $4$.
Then by the Chinese Remainder Theorem, the size of the subset of $\mathcal{X}$
divisible by $d$ odd is $x/(4d) + O(1)$.
The subset is empty if $d$ is even.

\begin{theorem} \label{thm:eulersieve}
Assume that $z \leq \log{x}$, and let $\mathcal{P}$ be the set of odd primes up to $z$.
Then
$$
S(\mathcal{X}, \mathcal{P}) = (1 + o(1)) \frac{e^{-\gamma} x}{2 \log{z}} \enspace .
$$
\end{theorem}
\begin{proof}
Let $P$ be the product of all odd primes up to $z$.
Applying the Legendre sieve, we have 
$$
S(\mathcal{X}, \mathcal{P}) = \sum_{d \mid P} \mu(d) \frac{x}{4d} + O(1)
= \frac{x}{4} \prod_{p \mid P} \left( 1 - \frac{1}{p} \right) + O(2^{\pi(z)})
= \frac{x}{2} \prod_{p \leq z} \left(1 - \frac{1}{p} \right) + O(2^{\pi(z)})  \enspace .
$$
With $z \leq \log{x}$, Corollary \ref{legendre_eq} gives the result.
\end{proof}

For some results we will need a stronger sieve, i.e. one where $z$ can grow 
larger than $\log{x}$.  The following special case of the Brun sieve adapted from 
 \cite[Section 3.2.3]{Greaves01} will suffice.  For sifting density 
 we use the simpler characterization found in \cite[Section 1.3.5]{Greaves01}.

\begin{theorem}[Brun sieve] \label{brun_sieve}
Let $\mathcal{P}$ be a set of primes all less than $z$ and let $d$ be a divisor of 
$\prod_{p \in \mathcal{P}} p$.
Assume $\mathcal{P}$ has sifting density $\kappa > 0$, i.e. 
there is a constant $A > 1$ such that
$$
\sum_{w \leq p < z} \frac{\log{p}}{p-1} \leq \kappa \log{ \left( \frac{z}{w} \right)} + A
\mbox{ when  $2 \leq w < z$ and $p \in \mathcal{P}$ } \enspace .
$$
Then 
$$
S(x,\mathcal{P}) \sim x \prod_{p \in \mathcal{P}} \left(1 - \frac{1}{p} \right)
$$
as $x \to \infty$, uniformly in $z \leq x^{1/(c \kappa \log\log{x})}$, where $c$ is an absolute constant.
\end{theorem}

As an application we give an upper bound on 
the count of $n \leq x$ divisible by only primes 
$\equiv 1 \bmod{4}$.
Despite being far from best-possible, it is adequate for our needs in a later proof.

\begin{corollary}\label{equiv_1}
The count of $n \leq x$ divisible by only primes $\equiv 1 \bmod{4}$
is $o(x/\log\log{x})$.
\end{corollary}
\begin{proof}
The count desired can be obtained by sieving all primes $p \equiv 3 \bmod{4}$, 
and if we restrict the set of sieving primes the count only gets larger.  So 
let $\mathcal{P}$ be the set of primes $p \leq x^{1/ (\log\log{x})^2}$
with $p \equiv 3 \bmod{4}$.  With $x$ large enough we have 
$x^{1/ (\log\log{x})^2} \leq x^{1/(c \kappa \log\log{x})}$ and the Brun sieve applies.
Then Corollary \ref{legendre_upper}
and Lemma \ref{reciprocalsums} yields
$$
x \prod_{p \in \mathcal{P}} \left( 1 - \frac{1}{p} \right)
\leq (1 + o(1)) x \cdot {\rm exp} \left( - \sum_{p \in \mathcal{P}} \frac{1}{p} \right)
\leq (1 + o(1)) x \cdot {\rm exp} \left( -\frac{c_2}{2} \log\log{x} \right) \enspace .
$$
\end{proof}

Finally, in Section \ref{sec:algorithms}  we will frequently use various measures 
for the average number of prime factors of a number.  The results in the following Lemma
are not new, but since we could not find a reference for the third equality we present a proof.

\begin{lemma} \label{thm:average}
Let $\omega(n)$ be the number of distinct prime factors of $n$
and $\Omega(n)$ be the total number of prime factors of $n$.  Let $p$ be a prime.  Then
\begin{align*}
& \sum_{n \leq x} \omega(n) = O(x \log\log{x}) \enspace ,\\
& \sum_{n \leq x} \Omega(n) = O(x \log\log{x}) \enspace, \mbox{and} \\
& \sum_{p \leq x} \Omega(p-1) = O\left( \frac{x \log\log{x}}{\log{x}} \right)\enspace .
\end{align*}
\end{lemma}
\begin{proof}
For the first two see \cite[Theorem 430]{HarWri79}.  If $\Omega$ is replaced by 
$\omega$ in the third statement, then Halberstam provided a proof in \cite{Hal56}.  
To prove the result above, 
it suffices to show that $\sum_p \Omega(p-1) - \omega(p-1)$ is $O(x/\log{x})$.

Let $i(q^r \mid p-1)$ be the indicator function for the event ``$q^r$ divides $p-1$".
Then 
\begin{equation} \label{indicator}
\sum_{p < x} \Omega(p-1) - \omega(p-1)
= \sum_{p < x} \sum_{q^r  < x \atop r \geq 2} i(q^r \mid p-1)
= \sum_{q^r  < x \atop r \geq 2} \sum_{p < x} i(q^r \mid p-1) \enspace .
\end{equation}
Focusing first on $q^r < \sqrt{x}$, $\sum_p i(q^r \mid p-1)$ is given by $\pi(x, q^r, 1)$.
Using Lemma \ref{bruntitch} that half is upper bounded by
$$
\sum_{q^r < \sqrt{x} \atop r \geq 2} \frac{2x}{\varphi(q^r) \log{(x/q^r)}}
\leq \sum_{q^r < \sqrt{x} \atop r \geq 2} \frac{2x}{\varphi(q^r) \log{\sqrt{x}}}
\leq \sum_{q^r < \sqrt{x} \atop r \geq 2} \frac{8x}{q^r \log{x}}  \enspace .
$$
This is $O(x/\log{x})$, for the sum over prime powers with power at least $2$ converges.
To see this, consider the terms for a given prime $q$.
With $r \geq 2$, those terms 
are bounded by the corresponding geometric series with value $\frac{1}{q(q-1)}$.
Then extending the sum to be over all integers gives
$$
 \sum_{2 \leq q < x} \frac{1}{q(q-1)}
\leq  \sum_{2 \leq q < x} \frac{2}{q^2}
= O(1) \enspace .
$$
Returning to the second half of (\ref{indicator}), we know that at most $x/q^r$ 
integers are multiples of $q^r$, so that half is upper bounded by 
$$
\sum_{\sqrt{x} < q^r < x \atop r \geq 2} \frac{x}{q^r}
\leq \sum_{\sqrt{x} < q^r < x \atop r \geq 2} \sqrt{x} \enspace .
$$
To count prime powers, we use
$$
\sum_{2 \leq r \leq \log{x} } \pi(x^{1/r})
\leq \pi(x^{1/2}) + \log{x} \cdot \pi(x^{1/3})
= O \left( \frac{x^{1/2}}{\log{x}} + \log{x} \cdot \frac{x^{1/3}}{\log{x}} \right)
$$
and so the second half is also $O(x/\log{x})$.
\end{proof}

\section{Two strong liars}

Denote by $A(x)$ the number of odd $n \leq x$ with 
$\prod_{p \mid n} \gcd(n', p')=1$.
Let $A_r(x)$ be the subset of $A(x)$ whose least prime dividing $n'$ is $r$.
Then $A(x) = \log_2{x}+ \sum_r A_r(x)$, where $\log_2{x}$ counts $n$ for which 
$n-1$ is a power of $2$.

We will break $\sum_r A_r(x)$ into three sums depending on whether 
$r < (\log\log{x})^{1-\epsilon}$, $(\log\log{x})^{1-\epsilon} \leq r \leq (\log\log{x})^{1+\epsilon}$, 
or $r > (\log\log{x})^{1+\epsilon}$.  
Call these, respectively, $\sum_1, \sum_2, \sum_3$.
For ease of notation we use 
$z_1$ for $(\log\log{x})^{1-\epsilon}$ and $z_2$ for $(\log\log{x})^{1+\epsilon}$, 
while $z$ will denote a generic bound on $r$.
Here and in the next section, we wish to prove that $\lim_{x \to \infty} f(x) = a$
where $f(x)$ is the quotient of our target function and a simpler approximation.
To prove a sequence $a_n$ has the limit $a$, it 
is sufficient to show that for every $\epsilon > 0$, 
$$
a - \epsilon \leq \liminf a_n \leq \limsup a_n \leq a + \epsilon \enspace .
$$

This strategy mirrors closely an argument from \cite{Erd48} (thanks to 
Carl Pomerance for help with a particularly perplexing point).  In fact, upper bounds 
on all three of $\sum_1, \sum_2, \sum_3$ are identical to those used by Erd\H{o}s.
However, the new definition of $A_r(x)$ required for the current work does 
necessitate a different approach for the lower bound to $\sum_3$.
A new writeup is useful for other reasons:  we have streamlined the discussion of 
prime reciprocal sums, clarified the derivation of the upper bound to $\sum_2$, and 
fixed several confusing typographical errors.

\begin{lemma} \label{Ar_sums}
Let $0 < \epsilon < 1$ and $z_1 = (\log\log{x})^{1 - \epsilon}$.
We have
$$
\sum_{r < z_1} A_r(x) = o_\epsilon \left( \frac{x}{\log\log{x}} \right)  \enspace .
$$
\end{lemma}
\begin{proof}
Suppose that $n$ is counted by $A_r(x)$ with $r$ an odd prime less than $z_1$.
Then $n \equiv 1 \bmod{r}$, but must not be divisible by any $p \equiv 1 \bmod{r}$.
So an upper bound on $A_r(x)$ is given by the count of $n$ not divisible by 
any $p \equiv 1 \bmod{r}$, and the count is further enlarged if we restrict our 
sieving set $\mathcal{P}$ to primes $p \equiv 1 \bmod{r}$ with $p < x^{1/(\log\log{x})^2}$.

Now the Brun sieve applies.  We use the upper bound from 
Corollary \ref{legendre_upper} and the unconditional lower bound from 
Lemma \ref{reciprocalsums} (note $z_1$ small enough so $r \leq \log{x} $).
For every $\epsilon$, we can take $x$ large enough so that
$$
A_r(x) \leq x \prod_{p \in \mathcal{P} } \left(1 - \frac{1}{p}\right)
\leq x \cdot {\rm exp}\left( - \sum_{p \in \mathcal{P}} \frac{1}{p} \right)
\leq x \cdot {\rm exp}\left( - \frac{\log\log{x}}{2\varphi(r)} \right) 
\leq x (\log\log{x})^{-3}
$$
where the last inequality follows from $r \leq (\log\log{x})^{1-\epsilon}$.
Then
$$
\sum_{r < z_1} A_r(x) \leq (\log\log{x})^{1-\epsilon} \cdot 
   	o \left( \frac{x}{(\log\log{x})^2} \right)
= o_\epsilon \left(\frac{x}{\log\log{x}} \right) \enspace .
$$
\end{proof}

\begin{lemma} \label{Ar_sums2}
Let $0 < \epsilon < 4/5$, $z_1 = (\log\log{x})^{1 - \epsilon}$, and 
$z_2 = (\log\log{x})^{1+\epsilon}$.  Then as $x \to \infty$
$$
\sum_{z_1 \leq r \leq z_2} \! A_r(x) \leq c \frac{\epsilon x}{\log\log\log{x}}
$$
where $c$ is an absolute constant.
\end{lemma}
\begin{proof}
For $z_1 \leq r \leq z_2$ we use a different upper bound on $A_r(x)$, namely the 
count of $n \leq x$ with $r$ as the smallest prime factor of $n-1$.  This is 
at most one away from the count of $n \leq x$ whose least prime factor is $r$.
By Corollary \ref{legendre_leastprime}, for large enough $x$
 this count is upper bounded by 
$$
c_1 \frac{x e^{-\gamma}}{r \log{r}} \enspace .
$$
Then 
$$
\sum_{r=z_1}^{z_2} \frac{c_1x e^{-\gamma}}{r \log{r}}
\leq \frac{c_1 x e^{-\gamma}}{ \log( (\log\log{x})^{1-\epsilon})} \sum_{r=z_1}^{z_2} \frac{1}{r}
\leq c' \frac{x e^{-\gamma}}{\log\log\log{x}} \cdot \log\left( \frac{1+\epsilon}{1-\epsilon} \right)
$$
where the sum is resolved via Lemma \ref{landau}.  Note 
$\epsilon < 4/5$ implies $\log(\frac{1+\epsilon}{1-\epsilon}) \leq 3\epsilon$.
\end{proof}

The final term is the one that will have the largest magnitude.  If $n$ is counted by 
$\sum_{r > z_2} A_r(x)$ then $n-1$ has no odd prime factor smaller than $z_2$.
We apply the Legendre sieve.

\begin{theorem}\label{thm:twostrongliars}
The number of odd $n \leq x$ with exactly two strong liars is given by 
$$
(1 + o(1)) \frac{x e^{-\gamma}}{\log\log\log{x}} \enspace .
$$
\end{theorem}
\begin{proof}
The main work is in counting odd $n \leq x$ with $\prod_{p \mid n} \gcd(n', p')=1$, 
and the main term is $\sum_{r > z_2} A_r(x)$.  This is smaller than the count of
$n$ where $n-1$ has no prime divisor smaller than $z_2$.  
With $z_2 < \log{x}$, the Legendre sieve gives us 
$$
\sum_{r > z_2} A_r(x) \leq
x \prod_{2 \leq r \leq z_2} \left( 1 - \frac{1}{r} \right) + 2^{\pi(z_2)}
= (1 + o(1)) \frac{x e^{-\gamma}}{\log{z_2}}
=  \frac{(1 + o(1)) x e^{-\gamma}}{(1 + \epsilon)\log\log\log{x}} \enspace .
$$

For a lower bound we exclude $n \equiv 1 \bmod{r}$ that are divisible by a prime 
$p \equiv 1 \bmod{r}$, and do this for all $r > z_2$.  For a given 
$r$ the number of $n$ excluded is 
$$
\sum_{p \equiv 1 \bmod{r}} \frac{x}{pr} + O(1)
$$
since the condition $n \equiv 0 \bmod{p}$ and $n \equiv 1 \bmod{r}$ repeats
every $pr$ integers by the Chinese Remainder Theorem. 
Note that $p \equiv 1 \bmod{r}$ and $pr \leq x$ implies that $r \leq \sqrt{x}$.
So applying Lemma \ref{reciprocalsums} (uniform upper bound)
and Lemma \ref{bruntitch} gives
\begin{align*}
\sum_{z_2<r < \sqrt{x}} & \sum_{p \equiv 1 \bmod{r}} \frac{x}{pr} + O(1)
\leq \sum_{z_2<r < \sqrt{x}}  \frac{x}{r} \frac{2}{\varphi(r)} (\log\log{x} + O(1))+
			O\left( \frac{x}{\varphi(r) \log{x}} \right)   \\
& \leq O\left( \sum_{z_2<r < \sqrt{x}} \frac{x\log\log{x} }{r^2} \right) + 
                   O\left( \sum_{z_2<r<\sqrt{x}} \frac{x}{r^2} \right) + 
		O\left( \frac{x\log\log{x}}{\log{x}} \right) \enspace .
\end{align*}
Taking the sum over integers rather than over primes, we have 
$$
\sum_{z_2 < r < \sqrt{x}} \frac{1}{r^2} < \int_{z_2}^{\sqrt{x}} \frac{1}{r^2} \ {\rm d}r
= -\frac{1}{\sqrt{x}} + \frac{1}{z_2}
< \frac{1}{(\log\log{x})^{1 + \epsilon}} \enspace .
$$
So the amount we are subtracting is upper bounded by
$$
O\left( \frac{x}{(\log\log{x})^{\epsilon}} \right) + O\left( \frac{x}{ (\log\log{x})^{1 + \epsilon}} \right)
+ O\left( \frac{x \log\log{x}}{\log{x}} \right)
= o_\epsilon \left( \frac{x}{\log\log\log{x}} \right) \enspace .
$$
Let $0 < \epsilon < 4/5$ be arbitrary.  By Lemma \ref{Ar_sums},
for large enough $x$ we have $\Sigma_1 < \epsilon$.  Then by Lemma \ref{Ar_sums2}
and the work above, we see that
$$
\frac{e^{-\gamma}}{1 + \epsilon} - \epsilon  \leq
\liminf \frac{A(x) }{x/\log\log\log{x}} \leq \limsup \frac{A(x)}{x/ \log\log\log{x}}
\leq \frac{e^{-\gamma}}{1 + \epsilon} + \epsilon +  c\epsilon \enspace .
$$
Since $0 < \epsilon < 4/5$ was arbitrary, the limit exists and the proper constant is 
indeed $e^{-\gamma}$.  As far as being divisible 
by at least one prime $\equiv 3 \bmod{4}$, by Corollary \ref{equiv_1} the number 
of $n \leq x$ only divisible by primes $\equiv 1 \bmod{4}$ is  
$o(x/\log\log{x})$.  The characterization in Proposition \ref{prop:n/p} now 
finishes the proof.
\end{proof}

The same proof technique can be extended to counting $n$ with exactly two Euler liars.

\begin{theorem}\label{thm:twoeulerliars}
The number of $n \leq x$ with exactly two Euler liars is given by 
$$
(1 + o(1)) \frac{x e^{-\gamma}}{2 \log\log\log{x}} \enspace .
$$
\end{theorem}
\begin{proof}
The characterization is given by Proposition \ref{prop:eulercases}; we start with the first case.
We use the same proof technique as that for Theorem \ref{thm:twostrongliars}.  For all 
terms except the main term, we can drop the condition that $n \equiv 3 \mod{4}$ at no loss.
It does affect the main term however:  by Theorem \ref{thm:eulersieve} the count of 
$n \leq x$ with $n \equiv 3 \mod{4}$ and not divisible by any factor less than $\log\log{x}$
is given by $(1 + o(1)) x/(2e^{\gamma}\log\log\log{x})$.

The second case is asymptotically smaller, since the number of $n \leq x$ with two prime
 factors is $O((x \log\log{x})/(\log{x}))$.
\end{proof}

\section{Two strong liars and two prime factors}

Our goal in this section is to prove the second of the three main theorems 
given in the introduction, thus providing an asymptotic formula for the count 
of odd $n \leq x$ with two strong liars and two prime factors.
Before discussing this in detail, we note that there are
$(1 + o(1)) (x \log\log x)/(\log x)$ numbers 
$n \le x$ that are a product of two primes
(this result is due to Landau \cite{Lan00}, see also
Wright \cite{Wri54}). 
The constant
$$
C = \prod_{p > 2} \left(1 - \frac{1}{(p-1)^2} \right)
$$
is what we would expect
from the following heuristic assumption: the two prime factors of $n$
are chosen independently, and fall into congruence classes in
the ``correct'' proportion.  The task, therefore, is to make
this rigorous.  The main idea of the proof will be to approximate
a count using a fixed number of terms of the inclusion-exclusion
formula, and then use a union bound to show that the approximate
count is good enough.  Hooley \cite{Hoo76} has called this strategy
the ``simple asymptotic sieve.''

\smallskip
In this section, $p$ and $q$ denote odd primes with
$p \le q$, and $d$ denotes a positive integer.  We now introduce
several sets:
$$
T = \{ pq \le x : \gcd(p',q') = 1 \};
$$
$$
T' = \{ pq \le x : \gcd(p',q') > 1 \};
$$
$$
S = \{ pq \le x : p,q \hbox{ odd } \};
$$
$$
S_d =  \{ pq \le x : p \equiv q \equiv 1 (d) \};
$$
$$
S^{(B)} = \{ n \in T : \gcd(p',q')>1 \hbox{ and has no primes $< B$ } \}.
$$
Note that if $d$ is odd,
$S_d = \{ pq \le x : d \hbox{ divides } \gcd(p',q') \}$.

Our first two tasks are to show that $S^{(B)}$ is not too large,
then to approximately count $S_d$.   A good tool for the first job
is the Brun-Titchmarsh theorem (Lemma \ref{bruntitch}), but the
 factor $\log(x/d)$ in its denominator can give trouble
when $d$ is close to $x$.  Our way around this is inspired by 
the chess player's gambit: give up a piece now to win later.  More
precisely, we will increase $x$, thereby bringing the log factor
under control at the price of a slightly worse upper bound
which is still good enough.
For the second job, since we will only be concerned with a fixed
number of $d$'s (depending on $B$), we can rely on a non-uniform
version of the prime number theorem for arithmetic progressions
(Lemma \ref{PNTprogs}).

\begin{lemma} \label{SBbound}
We have
$$
|S^{(B)}| = O \left( \frac{x \log\log x}{\sqrt B \log x} \right) \enspace ,
$$
where the implied constant is absolute.
\end{lemma}

\begin{proof}
Let $b$ be an odd prime.  We first find an upper bound for $|S_b|$, namely
$$
\sum_{\substack{p \le \sqrt x \\ p \equiv 1 (b)}}
\# \{ q : p \le q \le x/p \hbox{ and } q \equiv 1 (b) \}
\le \sum_{\substack{p \le \sqrt x \\ p \equiv 1 (b)}}
\# \{ q : q \le b^{1/2} x/p \hbox{ and } q \equiv 1 (b) \} \enspace .
$$
We may assume that $b \le \sqrt x$, since
the sum vanishes otherwise  (note that $b^2 \le pq = x$).
Then we are guaranteed that $b^{1/2}x / p > b$ since $p \sqrt b < pb \le pq \leq x$.
This allows us to estimate the summand using Lemma \ref{bruntitch},
and thereby get
$$
|S_b|
\le
\sum_{\substack{p \le \sqrt x \\ p \equiv 1 (b)}}
\frac{2 b^{1/2} x} {p \varphi(b) \log(x/(p b^{1/2}))} \enspace .
$$
We know that $p \le \sqrt x$ and $\sqrt b \le x^{1/4}$,
making $x / (p \sqrt b) \ge x^{1/4}$.  Therefore,
$$
\frac 1 {\log (x / (p \sqrt b))} \le \frac 4 {\log x} 
$$
and thus
$$
|S_b| \le
\frac {8 \sqrt b x} {\varphi(b) \log x}
\sum_{\substack{p \le \sqrt x \\ p \equiv 1 (b)}} \frac 1 p \enspace .
$$
By Lemma \ref{reciprocalsums}, the inner sum has an upper bound of 
$2 \varphi(b)^{-1}(\log\log{x} + O(1))$.
Summing over
all primes $b \ge B$, we get the result.
\end{proof}

\begin{lemma} \label{Sdasympt}
Let $d \ge 1$.  Then
$$
|S_d| = (1 + o_d(1)) \frac{x \log\log x}{\varphi(d)^2 \log x} \enspace .
$$
\end{lemma}

\begin{proof}
We have
$$
|S_d| = \sum_{\substack{p \le \sqrt x \\ p \equiv 1 (d)}}
\# \{ q : p \le q \le x/p \hbox{ and } q \equiv 1 (d) \} \enspace .
$$
If we drop the lower bound on $q$, we incur an error that is
no more than
$$
\sum_{p \le \sqrt x }
\# \{ q : q \le p \}
\le
\sum_{p \le \sqrt x }
\# \{ q : q \le \sqrt x \}
=
O \left(\frac x {\log^2 x} \right) \enspace .
$$
Accordingly, we can work with the simpler sum
\begin{equation}\label{simpler}
\sum_{\substack{p \le \sqrt x \\ p \equiv 1 (d)}}
\# \{ q : q \le x/p \hbox{ and } q \equiv 1 (d) \} \enspace .
\end{equation}
Fix $\epsilon$ with $0 < \epsilon < 1/2$.   We will split the
sum, using the break point $p = x^\epsilon$.
The contribution to (\ref{simpler}) from the $p \le x^\epsilon$ is
\begin{equation}\label{smallp}
\sum_{\substack{p \le x^\epsilon \\ p \equiv 1 (d)}}
\# \{ q : q \le x/p \hbox{ and } q \equiv 1 (d) \}
= \sum_{\substack{p \le x^\epsilon \\ p \equiv 1 (d)}}
(1 + o_d(1)) \frac{x}{\varphi(d) p \log (x/p)}
\end{equation}
by Lemma \ref{PNTprogs}.
The assumption $p \leq x^{\epsilon}$ implies $x/p \geq x^{1 - \epsilon}$.
Thus $(1-\epsilon)\log x \le \log(x/p) \le \log x$, which gives
$$
\frac{((1 + o_d(1))x}{\varphi(d) \log x}
\sum_{\substack{p \le x^\epsilon \\ p \equiv 1 (d)}} \frac 1 p
\le \hbox{(\ref{smallp})} \le
\frac{((1 + o_d(1))x}{(1 - \epsilon) \varphi(d) \log x}
\sum_{\substack{p \le x^\epsilon \\ p \equiv 1 (d)}} \frac 1 p \enspace .
$$
For the prime reciprocal sum we apply Lemma \ref{reciprocalsums}, which yields
$$
\sum_{\substack{p \le x^\epsilon \\ p \equiv 1 (d)}} \frac 1 p
= (1 + o_d(1)) \frac{\log\log x^\epsilon}{\varphi(d)}
= (1 + o_d(1)) \frac{\log\log x + \log \epsilon}{\varphi(d)} \enspace .
$$
This gives
$$
(1 + o_d(1))\frac{x(\log\log x + \log \epsilon)}
                      {\varphi(d)^2 \log x}
\le \hbox{(\ref{smallp})} \le
(1 + o_d(1))\frac{x(\log\log x + \log \epsilon)}
                      {(1 - \epsilon) \varphi(d)^2 \log x}
$$
and since $\epsilon$ was arbitrary and $x \to \infty$, we conclude
$$
(\ref{smallp}) = (1 + o_d(1)) \frac{x \log\log x}{\varphi(d)^2 \log x} \enspace .
$$

To finish off the proof, we will show that the sum over
primes larger than $x^\epsilon$ does not grow this quickly.
This sum is
\begin{equation}\label{largep}
\sum_{\substack{x^\epsilon < p \le \sqrt x \\ p \equiv 1 (d)}}
\# \{ q : q \le x/p \hbox{ and } q \equiv 1 (d) \}
\end{equation}
and it has an upper bound of
$$
\sum_{x^\epsilon < p \le \sqrt x }
\# \{ q : q \le x/p \}
=
\sum_{x^\epsilon < p \le \sqrt x }
(1 + o(1)) \frac x {p \log (x/p)}
\le
(1 + o(1)) \frac x {(1/2) \log x}
\sum_{x^\epsilon < p \le \sqrt x } \frac 1 p \enspace .
$$
By Lemma \ref{landau}
the inner sum is $\log(1/2) - \log \epsilon + O((\log x)^{-1})$.
So
$$
\hbox{(\ref{largep})} = O_\epsilon \left( \frac x {\log x} \right)
$$
and we are done.
\end{proof}

We are now ready for the main event.  Our strategy will be to 
first estimate how many odd $pq$'s satisfy only condition (2) in
the characterization of strong liars in Proposition \ref{prop:n/p}
(this is of interest by itself).  Then, we make a similar estimate
under the additional requirement that at least one of $p$ and $q$ be 3 modulo 4.

In the next two theorems,
$$
C := \prod_{p > 2} \left(1 - \frac{1}{(p-1)^2}\right) = 0.66016...\ .
$$

\begin{theorem} \label{thm:condition2}
The number of odd $n = pq \le x$ with $\gcd(p',q')=1$ is
\begin{equation}\label{maineqn}
(1 + o(1)) \frac{C x \log\log x} {\log x}.
\end{equation}
\end{theorem}

\begin{proof}
We have
$$
\bigcup_{\substack{2<r<B \\ r \  \rm prime}} S_d \subseteq T' \subseteq
\bigcup_{\substack{2<r<B \\ r \  \rm prime}} S_d \cup S^{(B)} \enspace .
$$
The set of odd primes up to $B$ is finite, so we can use
inclusion-exclusion and get
$$
\sum_r |S_r| - \sum_{r,s} |S_{rs} | + \sum_{r,s,t} |S_{rst} | + \cdots
\le |T'| \le
\sum_r |S_r| - \sum_{r,s} |S_{rs} | + \sum_{r,s,t} |S_{rst} | + \cdots
+ |S^{(B)}| \enspace ,
$$
where $r,s,t,\ldots$ denote distinct odd primes $< B$.  This can be
written another way as
\begin{equation} \label{musums}
 \sum_{\substack{\rm odd \ d \ge 3 \\  B-\rm smooth}} -\mu(d) |S_d|
\le |T'| \le
 \sum_{\substack{\rm odd \ d \ge 3 \\  B-\rm smooth}} -\mu(d) |S_d|
+ |S^{|B|}| \enspace ,
\end{equation}
where the sums over $d$ are finite since $S_d$ is empty for $d > x$.
Since $S$ is the disjoint union of $T$ and $T'$, and $S = S_1$,
we have
$$
|T| = |S| - |T'| = |S_1| - |T'| \enspace .
$$
Combining this with (\ref{musums}) we get
$$
\sum_{\substack{\rm odd \ d \ge 1 \\  B-\rm smooth}} \mu(d) |S_d|
- |S^{(B)}|
\le |T| \le
\sum_{\substack{\rm odd \ d \ge 1 \\ B-\rm smooth}} \mu(d) |S_d| \enspace .
$$
Since $B$ is fixed, we can use Lemma \ref{Sdasympt}
to express the sum over $d$ as
$$
(C_B + o_B(1)) \frac{x \log\log x}{\log x} \enspace .
$$
where
$$
C_B = \prod_{3 \le p < B} \left(1 - \frac {1}{(p-1)^2}\right)
    = \sum_{\substack{\rm odd \ d \ge 1 \\  B-\rm smooth}}
           \frac{\mu(d)}{\varphi(d)^2} \enspace .
$$
Combining these results with Lemma \ref{SBbound}, we then get
$$
\frac{|T|}{x \log\log x / \log x}
= C_B + o_B(1) + O(B^{-1/2}) \enspace .
$$
For any $\delta>0$, we can choose a $B$ for which
both $C_B - C$ and the $B^{-1/2}$ term are bounded by $\delta/4$
in absolute value.  With any such choice of $B$, the
$o_B(1)$ term will be no more than $\delta/2$ for sufficiently
large $x$, so
$$
C - \delta
\le
\lim\inf \frac{|T|}{x \log\log x / \log x}
\le
\lim\sup \frac{|T|}{x \log\log x / \log x}
\le
C + \delta \enspace .
$$
Since $\delta$ is arbitrary, we conclude that
the limit as $x \rightarrow \infty$ exists and equals $C$.
\end{proof}

\begin{theorem} \label{thm:1mod4}
The number of odd $n = pq \le x$ with $|S(n)| = 2$ is
\begin{equation}\label{maineqn41}
(1 + o(1)) \frac{3 C \cdot x \log\log x} {4 \log x} \enspace .
\end{equation}
\end{theorem}

\begin{proof}
By Proposition \ref{prop:n/p} we need to count odd $n \leq x$ with 
$n = pq$, $\gcd(p',q')=1$, and at least one of $p,q$ congruent to $3$ modulo $4$.
Since we have the count of $n \leq x$ with two prime factors and $\gcd(p',q')=1$, 
it suffices to subtract those where $p \equiv q \equiv 1 \mod{4}$.

This is very similar to the proof of Theorem \ref{thm:condition2},
so we only note the differences.  First, the ``universe'' $S$ is
no longer $S_1$ but $S_4$.  Second, we define
$T_{4,1}$ and $T'_{4,1}$ similarly to $T$ 
and $T'$, but with the additional requirement that 
$p \equiv q\equiv 1$ mod 4.  Then, as before, a combinatorial
argument gives
$$
\sum_{\substack{\rm odd \ d \ge 1 \\  B-\rm smooth}} \mu(d) |S_{4d}|
- |S^{(B)}|
\le |T_{4,1}| \le
\sum_{\substack{\rm odd \ d \ge 1 \\  B-\rm smooth}} \mu(d) |S_{4d}| \enspace ,
$$
where we re-use $S^{(B)}$ since removing $b$'s that are 3 mod 4
only makes the lower bound larger.
By Lemma \ref{Sdasympt} we have
$$
\sum_{\substack{\rm odd \ d \ge 1 \\  B-\rm smooth}} \mu(d) |S_{4d}|
=
\left(
\frac 1 4 \prod_{3 \le p < B} \left(1 - \frac {1}{(p-1)^2}\right)
+ o_B(1) \right) \frac{x \log\log x}{\log x} \enspace .
$$
The rest of the proof proceeds just as for Theorem \ref{thm:condition2}.
\end{proof}

In the table below, count 1 is the number of $n \leq x$ with $n=pq$
and $\gcd(p',q')=1$.  Count 2 adds the condition that $p \equiv q \equiv 1 \mod{4}$.
As can be seen from the first two
columns, Theorem \ref{thm:condition2} is reasonably accurate,
despite the slowly growing $\log\log{x}$ factor.  In the last two columns the asymptotic
expression can be seen to be a bit of an overestimate.  
We believe this reflects ``Chebyshev's bias,'' whereby 
the residue class 1 mod 4 gets, among small primes, noticeably
less than its fair share.
\bigskip

\begin{tabular}{|r| r | r | r | r |}
$x$ & count 1 & prediction & count 2 & prediction \\ \hline
$10^3$     & 166     & 184.70      &  28     &  46.17      \\
$5 \times 10^3$      & 795     & 830.16      &  149    &  207.54     \\
$10^4$    & 1544    & 1591.44     &  298    &  397.86     \\
$5 \times 10^4$    & 7246    & 7264.91     &  1473   &  1816.23    \\
$10^5$  & 14027   & 14011.09    &  2872   &  3502.77    \\
$5 \times 10^5$   & 65442   & 64754.58    &  13681  &  16188.65   \\
$10^6$  & 127207  & 125471.12   &  26792  &  31367.78   \\
$5 \times 10^6$  & 595382  & 585478.01   &  126898 &  146369.50  \\
$10^7$ & 1159409 & 1138603.46  &  248242 &  284650.87  \\
$5 \times 10^7$ & 5459378 & 5353378.05 & 1178844 & 1338344.51 \\
  $10^8$ & 10653388 & 10441331.16 & 2307619 & 2610332.79 \\
 $ 5 \times 10^8$ & 50424160 & 49392155.46 & 10991685 & 12348038.86 \\
  $10^9$ & 98596968 & 96563937.17 & 21542038 & 24140984.29 \\
\end{tabular}

\section{Tabulation algorithm} \label{sec:algorithms}

Switching gears, in this section we design and analyze an algorithm that tabulates 
all $n \leq x$ with exactly two strong liars, thus giving an exact count.  This appears 
to be unstudied, so we start with naive ideas and improve upon them.

First, we address the costs of basic operations.  We can add two numbers with $k$ 
bits using $O(k)$ bit operations, and we use $M(k)$ to denote the cost of multiplying two 
$k$ bit numbers.  A classic fast multiplication algorithm is that of Sch\"{o}nhage and Strassen
 with $M(k) = O(k \log{k} \log\log{k})$, but it has been recently superseded by \cite{Fur09}.
A good discussion along with 
a table comparing different multiplication algorithms may be found 
in \cite[Section 8.3]{GathGer03}.  Finally,  for integers of $k$ bits the best gcd algorithm
 takes $O(M(k)\log{k})$ bit operations \cite{SteZim04}.

Turning to tabulation algorithms, 
a truly naive method would be to consider each $n$ in turn by factoring and then 
applying Monier's formula.  Since factoring is expensive for an individual $n$ 
but has a cheap amortized cost when factoring a range of $n$, we instead generate 
all factorizations first before applying Proposition \ref{prop:monier}.

To factor all positive integers $n \leq x$, we will generate an array where the largest 
prime factor of $n$ is stored at index $n$.  To do so, initialize the array with all zeros.
Starting with $p=2$, let $p$ be the next largest index whose value is $0$.  Then
take all indices that are a multiple of $p$ and overwrite the value with $p$.  Do this for 
all $p \leq \sqrt{x}$.  Each operation is an addition, and the total number of operations is 
$$
\sum_{p \leq \sqrt{x}} \frac{x}{p} = O(x \log\log{x}) \enspace ,
$$
making the total complexity $O(x \log{x})$ space and $O(x (\log{x}) (\log\log{x}))$ bit operations.
One can retrieve the factorization of $n$ at an amortized cost of $M(\log{n}) \log\log{n}$
by dividing $n$ by $p$ and then recursively looking up the largest prime factor of $n/p$
in the table.  The average of $\log\log{n}$ for the number of prime factors of $n$
comes from Lemma \ref{thm:average}.

This then gives Algorithm 0:  factor all integers $n \leq x$, then apply Monier's formula 
to each $n$.  Generating the array with the largest prime factor of each $n \leq x$ costs 
$O(x (\log{x}) (\log\log{x}))$ bit operations and uses $O(x \log{x})$ space.
We then have  a $\gcd$ check
 for each distinct prime divisor of $n$, 
which by Lemma \ref{thm:average} is a total of $O(x \log\log{x})$ $\gcd$'s at a total cost 
of $O(x M(\log{x}) (\log\log{x})^2)$ bit operations.  The total cost of the factorization 
retrievals is $O(x M(\log{x}) (\log\log{x}))$ by the same theorem, and doesn't affect 
the asymptotic running time.

Our first improvement will be to reduce the number of $\gcd$ checks.
For all $p \mid n$, form $\ell = \prod_{p \mid n} p'$.  Then checking Monier's formula
only requires a single $\gcd$ application.  This improvement is implemented 
in Algorithm \ref{alg:naive1}.

\begin{algorithm} \label{alg:naive1}
\caption{Naive tabulation}
Factor all $n \leq x$ using a sieve \;
\tcc{For each $n$ build $p'$ product and gcd with $n'$}
\For{ $n \leq x$}{
  $\ell = 1$ \;
  \For{ $p \mid n$}{
    $\ell \leftarrow \ell \cdot p'$ \;
    }
    \If{ $\gcd(\ell, n') \neq 1$}{ set $0$}
    \If{ $n$ odd and all $p \equiv 1 \bmod{4} $}{set $0$}{set $1$}
  }
\end{algorithm}

Note that the tabulation includes even $n$ with $\prod_{p \mid n} \gcd(p', n') = 1$, but 
it is trivial to isolate the odd survivors if required. 

\begin{theorem}\label{alg1_analysis}
Algorithm \ref{alg:naive1} stores at most $O(x \log\log{x})$ integers $\leq x$ and runs using
$O(x \cdot M(\log{x}) \log\log{x})$ bit operations.
\end{theorem}
\begin{proof}
As discussed the factoring step costs 
$O(x (\log{x})( \log\log{x}))$ time and $O(x \log{x})$ space to generate the array and 
 $O(x M(\log{x}) (\log\log{x}))$ bit operations to generate all the factorizations 
 over the course of the algorithm.

Algorithm \ref{alg:naive1} then does a multiplication for every distinct prime 
divisor of every $n \leq x$, a total of $O(x \log\log{x})$ multiplications 
by Lemma \ref{thm:average}.  The algorithm also does  $x$ gcd 
computations at a cost of  $O(x M(\log{x}) (\log\log{x}))$ bit operations, 
and the multiplications have the same total cost.
\end{proof}

In developing a better tabulation algorithm we seek to use more of a sieve strategy.
Note that for a given prime $p$ with $p-1$ divisible by an odd prime $r$, we can 
cross off any $n = p \cdot d$ where $d \equiv 1 \bmod{r}$.  For in this case 
$\gcd(p', (n/p)') \neq 1$ and thus $|S(n)| > 2$ by Proposition \ref{prop:n/p}.
Such $n$ are exactly those in the arithmetic progression 
$$
\{ n = p + kpr \ : \ k \in \N\} \enspace .
$$

In Algorithm \ref{alg:sieve}, checking all prime factors for one that is congruent to 
$3$ modulo $4$ would be too expensive.  Thus we add another sieving step, and 
introduce three states for each integer.  Integers start out labeled ``2."  If they fail 
to have $\prod_{p \mid n} \gcd(p',n')=1$ they get labeled ``0".  Finally, those 
divisible by a prime congruent to $3$ modulo $4$ are labeled ``1" and counted.

\begin{algorithm} \label{alg:sieve}
\caption{Sieving tabulation}
Generate an array with the largest prime factor of $n$ for all $n \leq x$\;
Initialize new array with $2$ in each odd entry, $0$ in each even \;
\For{ primes $p \leq x$ }{
   generate factorization of $p-1$ \;
  \For{ odd prime $r \mid (p-1)$}{
        \For{ $n \equiv p \bmod{pr}$}{ set $0$ \;}
    }
  }
\For{ primes $p \leq x$ with $p \equiv 3 \mod{4}$}{
 \For{ multiples of $p$ with value $2$}{
   set $1$ \;
 }
}
count $n$ with value $1$ \;
\end{algorithm}

\begin{theorem} \label{alg2_analysis}
Algorithm \ref{alg:sieve} stores at most $O(x)$ integers $\leq x$ and 
runs using $O(x (\log{x})( \log\log{x}))$ bit operations.
\end{theorem}
\begin{proof}
Generating the array of largest prime factors takes $O(x (\log{x})(\log\log{x}))$ bit operations 
and $O(x \log{x})$ space, as does the final sieving at line (8).  The main difficulty of 
the algorithm is the loop at line (3).

With the array of largest prime factors in hand, identifying primes is easy.  Then the main 
loop has two components.  The first is generating the factorization of $p-1$ for all primes 
up to $x$.  Since $\sum_{p \leq x} \Omega(p-1) = O(x \log\log{x}/\log{x})$
by Lemma \ref{thm:average}, the total cost in bit operations is 
$$
O\left( \frac{x M(\log{x}) \log\log{x}}{\log{x}} \right) \enspace .
$$
Even using a naive multiplication algorithm with $M(\log{x}) = O(\log{x})^2$, this is no 
worse than $O(x (\log{x}) (\log\log{x}))$.

The second component of the main loop involves checking each element 
of the sequence $n = p + kpr$, where $p$ runs over primes up to $x$ and $r$ 
runs over the distinct prime divisors of $p-1$.  Generating such a sequence 
requires $x/(pr)$ additions, making the total number of additions
$$
\sum_{p \leq x} \sum_{r \mid p-1} \frac{x}{pr} \enspace .
$$
To evaluate this sum, we reverse the order of summation.
This same sum appeared in the proof of Theorem \ref{thm:twostrongliars}; 
note that once again $r \mid p-1$ and $pr \leq x$ implies $r \leq \sqrt{x}$.
We have
$$
x \sum_{r \leq x} \frac{1}{r} \sum_{p \equiv 1 \bmod{r}} \frac{1}{p} 
\leq O\left(x \sum_{r \leq \sqrt{x}} \frac{1}{r} \frac{\log\log{x}}{\varphi(r)} \right)
= O\left( x \log\log{x} \sum_{r \leq \sqrt{x}} \frac{1}{r^2} \right) \enspace .
$$
The sum over primes has a 
constant upper bound.
Thus there are $O(x \log\log{x})$ additions at a cost of $O( x (\log{x})(\log\log{x}))$ 
bit operations.
\end{proof}

Algorithm \ref{alg:sieve} was implemented, giving the following counts of composite 
$n \leq x$ with exactly two strong liars.
\bigskip

\begin{tabular}{|c| r | c |}
$x$ & $n \leq x$ with $|S(n)| = 2$ & $\rm{count} \cdot \frac{\log\log\log{x}}{x}$ \\ \hline
$10^3$ &    243                    & 0.1601\\
$10^4$ &    2553                  & 0.2036 \\
$10^5$ &    25955                &  0.2319 \\
$10^6$ &    261280               &   0.2522 \\
$10^7$ &    2616237            &     0.2675    \\        
$10^8$ &    26140023          &      0.2795  \\       
$10^9$ &    260899381          &     0.2893 
\end{tabular}
\bigskip

With $e^{-\gamma} \approx 0.5615$, we see that convergence to the asymptotic 
formula is quite slow.  As Daniel Shanks
once wrote, $\log\log\log x$ does go to infinity, but ``with great
dignity."

\section{Conclusions and future work}

It is interesting that counts of $n \leq x$ with $\gcd(n, \varphi(n)) = 1$, with 
$|F(n)| = 1$, and with $|S(n)| = 2$ all have the same asymptotic formula, and it 
suggests that there might be some general class of arithmetic sets whose size 
can be approximated by the set of $n$ with no prime factor smaller than $\log\log{x}$.
The set of $n$ satisfying the best case for the Lucas pseudoprime test
would be well worth studying next.
It is worth remarking that the three sets
$\{n \leq x \ : \ \gcd(n, \varphi(n))=1\}, \{n \leq x \ : \ |F(n)|=1\}, \{n \leq x \ : \ |S(n)|=2\}$
are not the same.  For $9$ has two strong liars, 
but $\gcd(9, \varphi(9)) \neq 1$ and $9$ has two Fermat liars rather than one.  Also, 
$15$ satisfies $\gcd(15, \varphi(15)) = 1$, but $15$ has more than one Fermat liar.

The authors of \cite{ErdPom86} give a number of other results regarding the size of 
$|F(n)|$, and it would be worth extending those results to $|S(n)|$ and $|E(n)|$.
Our Theorem \ref{thm:1mod4} is in a different vein, and it would be nice to extend 
it to counts of $n$ with two strong liars and $k$ prime factors for $k > 2$.

The slow rate of convergence of exact counts of $n$ with two strong liars to the 
asymptotic formula cries out for a more precise formula with a faster rate of 
convergence.  It seems that a large part of the error comes from the fact that 
the Mertens bound is not very accurate when one only sieves by small primes.
As for why the count of $n$ with two strong liars is approximated by sieving up 
to $\log\log{x}$, consider the following heuristic argument.  A typical $n$ will 
have $\log\log{n}$ prime factors $p$.  For a given prime $r$ of size roughly $\log\log{n}$, 
the expected number of $p$ with $r \mid p-1$ is one.  Since we need $n-1$ to not 
be divisible by $r$, we exclude all the $n\leq x$ with $n-1$ divisible by a prime less 
than $\log\log{x}$.

Though asymptotically the number of $n$ with $|S(n)| = 2$ is density $0$, 
for quite some time the proportion is more than a quarter of all integers.
It would be worth knowing at what point the proportion is less than an arbitrary 
constant $0 < c < 1$, as well as the proportion of $n$ with $|S(n)| = \ell$ 
for values of $\ell$ greater than two.  

Our algorithm counts odd $n \leq x$ with two strong liars by tabulating them.
If Theorem \ref{thm:twostrongliars} could be improved by finding an explicit error bound, 
one could find an approximate count much faster through the use of that formula.

\bibliography{miller_rabin}

\providecommand{\bysame}{\leavevmode\hbox to3em{\hrulefill}\thinspace}
\providecommand{\MR}{\relax\ifhmode\unskip\space\fi MR }
\providecommand{\MRhref}[2]{%
  \href{http://www.ams.org/mathscinet-getitem?mr=#1}{#2}
}
\providecommand{\href}[2]{#2}
\begin{thebibliography}{10}

\bibitem{BateHorn62}
Paul~T. Bateman and Roger~A. Horn, \emph{A heuristic asymptotic formula
  concerning the distribution of prime numbers}, Math. Comp. \textbf{16}
  (1962), 363--367.

\bibitem{DamLadPom93}
Ivan Damg{\aa}rd, Peter Landrock, and Carl Pomerance, \emph{Average case error
  estimates for the strong probable prime test}, Math. Comp. \textbf{61}
  (1993), no.~203, 177--194.

\bibitem{Erd48}
Paul Erd{\H{o}}s, \emph{Some asymptotic formulas in number theory}, J. Indian
  Math. Soc. (N.S.) \textbf{12} (1948), 75--78.

\bibitem{ErdPom86}
Paul Erd{\H{o}}s and Carl Pomerance, \emph{On the number of false witnesses for
  a composite number}, Math. Comp. \textbf{46} (1986), no.~173, 259--279.

\bibitem{Fur09}
Martin F{\"u}rer, \emph{Faster integer multiplication}, SIAM J. Comput.
  \textbf{39} (2009), no.~3, 979--1005.

\bibitem{GranPom02}
Andrew Granville and Carl Pomerance, \emph{Two contradictory conjectures
  concerning {C}armichael numbers}, Math. Comp. \textbf{71} (2002), no.~238,
  883--908.

\bibitem{Greaves01}
George Greaves, \emph{Sieves in {N}umber {T}heory}, Ergebnisse der Mathematik
  und ihrer Grenzgebiete (3) [Results in Mathematics and Related Areas (3)],
  vol.~43, Springer-Verlag, Berlin, 2001.

\bibitem{Hal56}
H.~Halberstam, \emph{On the distribution of additive number-theoretic
  functions. {III}}, J. London Math. Soc. \textbf{31} (1956), 14--27.

\bibitem{HarWri79}
G.~H. Hardy and E.~M. Wright, \emph{An {I}ntroduction to the {T}heory of
  {N}umbers}, fifth ed., The Clarendon Press Oxford University Press, New York,
  1979.

\bibitem{Hoo76}
C.~Hooley, \emph{Applications of {S}ieve {M}ethods to the {T}heory of
  {N}umbers}, Cambridge University Press, Cambridge, 1976, Cambridge Tracts in
  Mathematics, No. 70.

\bibitem{Lan00}
Edmund Landau, \emph{Sur quelques probl\`{e}mes relatifs \`{a} la distribution
  des nombres premiers}, Bull. Soc. Math. France \textbf{28} (1900), 25--38,
  reprinted in \emph{Collected Works}, Vol. 1, pp.92--105.

\bibitem{Lan09}
\bysame, \emph{Handbuch der {L}ehre von der {V}erteilung der {P}rimzahlen. 2
  {B}\"ande}, Chelsea Publishing Co., New York, 1953, 2d ed, With an appendix
  by Paul T. Bateman.

\bibitem{Mon80}
Louis Monier, \emph{Evaluation and comparison of two efficient probabilistic
  primality testing algorithms}, Theoret. Comput. Sci. \textbf{12} (1980),
  no.~1, 97--108.

\bibitem{MonVau73}
H.~L. Montgomery and R.~C. Vaughan, \emph{The large sieve}, Mathematika
  \textbf{20} (1973), 119--134.

\bibitem{SteZim04}
Damien Stehl{\'e} and Paul Zimmermann, \emph{A binary recursive gcd algorithm},
  Algorithmic number theory, Lecture Notes in Comput. Sci., vol. 3076,
  Springer, Berlin, 2004, pp.~411--425.

\bibitem{GathGer03}
Joachim von~zur Gathen and J{\"u}rgen Gerhard, \emph{Modern {C}omputer
  {A}lgebra}, second ed., Cambridge University Press, Cambridge, 2003.

\bibitem{Wal36}
Arnold Walfisz, \emph{Zur additiven {Z}ahlentheorie. {II}}, Math. Z.
  \textbf{40} (1936), no.~1, 592--607.

\bibitem{Wri54}
E.~M. Wright, \emph{A simple proof of a theorem of {L}andau}, Proc. Edinburgh
  Math. Soc. (2) \textbf{9} (1954), 87--90.

\end{thebibliography}
\bibliographystyle{amsplain}

\end{document}